\documentclass[12pt,reqno]{article}
\usepackage{enumerate}
\usepackage{fullpage,hyperref,bm}
\usepackage{amsmath, amssymb, amsthm, url, graphicx,rotating,tikz,tkz-graph,relsize}
\usepackage{relsize}
\usepackage{mathrsfs}
\usepackage{dsfont}
\usepackage{ulem}

\hypersetup{
    colorlinks,
    linkcolor={red!50!black},
    citecolor={blue!50!black},
    urlcolor={blue!80!black}
}

\let\emph\relax
\DeclareTextFontCommand{\emph}{\itshape}

\newtheoremstyle{plainsl}%
        {\topsep}
        {\topsep}
        {\slshape} 
        {}
        {\normalfont\bfseries}
        {.}
        { }
        {}

\theoremstyle{plainsl}
\newtheorem{theorem}{Theorem}[section]
\newtheorem{proposition}[theorem]{Proposition}
\newtheorem{lemma}[theorem]{Lemma}
\newtheorem{corollary}[theorem]{Corollary}
\newtheorem{definition}[theorem]{Definition}
\newtheorem{example}[theorem]{Example}
\newtheorem{problem}[theorem]{Problem}

\newtheorem{remark}[theorem]{Remark}

\numberwithin{equation}{section}

\newcommand{\PropM}[1]{\Phi_{#1}}

\newcommand{\re}{\mathbb{R}}
\newcommand{\cx}{\mathbb{C}}
\newcommand{\ints}{\mathbb{Z}}
\newcommand{\rats}{\mathbb{Q}}
\newcommand{\EE}{\mathbb{E}}
\newcommand{\FF}{\mathbb{F}}
\newcommand{\KK}{\mathbb{K}}

\newcommand{\BMA}{\mathbb{A}}
\newcommand{\BMB}{\mathbb{B}}
\newcommand{\cA}{\mathcal{A}}

\newcommand{\cB}{\mathcal{B}}
\newcommand{\cC}{\mathcal{C}}

\newcommand{\cE}{{\mathcal E}}  

\newcommand{\cO}{{\mathcal O}}

\newcommand{\cQ}{{\mathcal Q}}
\newcommand{\cR}{{\mathcal R}}

\newcommand{\cS}{{\mathcal S}}
\newcommand{\sT}{{\mathsf{T}}}

\newcommand{\be}{{\mathbf e}}
\newcommand{\bx}{{\mathbf x}}

\newcommand{\ones}{{\mathds{1}}}  

\newcommand{\Mat}{\mathsf{Mat}}

\newcommand{\tr}{\mathsf{tr}}

\newcommand{\mmod}{ \  {\scriptscriptstyle{\rm{MOD}}} \  }

\DeclareMathOperator\Sym{Sym}

\DeclareMathOperator\rk{rk}
\DeclareMathOperator\spn{span}
\DeclareMathOperator\Gal{\mathsf{Gal}}
\DeclareMathOperator\GL{\mathsf{GL}}
\DeclareMathOperator\Dic{\mathsf{Dic}}

\begin{document}
\thispagestyle{empty}
\setcounter{page}{1}
\title{Rational Delsarte designs and Galois fusions of association schemes}
\author{
Jesse Lansdown \\
School of Mathematics and Statistics \\
University of Canterbury, New Zealand\\
{\tt jesse.lansdown@canterbury.ac.nz} \\
William J.~Martin \\
Department of Mathematical Sciences \\
Worcester Polytechnic Institute, 
Worcester, MA USA \\
{\tt  martin@wpi.edu}}

\date{\today} 
\maketitle

\medskip

\begin{abstract}
Delsarte theory, more specifically the study of codes and designs in association schemes, has proved invaluable in studying an increasing assortment of association schemes in recent years. Tools motivated by the study of error-correcting codes in the Hamming scheme and combinatorial $t$-designs in the Johnson scheme apply equally well in association schemes with irrational eigenvalues.   We assume here that we have a commutative association scheme with irrational eigenvalues and wish to study its Delsarte $T$-designs. We explore when a $T$-design is also a $T'$-design where $T'\supseteq T$ is controlled by the orbits of a Galois group related to the splitting field of the association scheme.  We then study Delsarte designs in the association schemes of finite groups, with a detailed exploration of the dicyclic groups.
\end{abstract}


\noindent {\bf 2020 MSC Subject Codes:} 05E30, 05B30,  20C15, 16S50.

%
%
\section{Introduction}
\label{Sec:intro}

The thesis of Philippe Delsarte \cite{del} was a landmark in coding theory and combinatorial 
design theory.  Employing the language of  association schemes, Delsarte's approach was 
one of the first to cast combinatorial questions in a linear algebraic framework where tools
from matrix theory, orthogonal polynomials, and optimisation could be applied. In particular, 
Delsarte \cite{del} characterised orthogonal arrays and block designs as 01-vectors orthogonal 
to specific eigenspaces of the Hamming scheme and Johnson scheme, respectively. It is 
natural, then, to look at analogous substructures in the $q$-analogues of these families of 
association schemes and other classical families of $Q$-polynomial distance-regular graphs \cite{bcn}. 
The following decades brought more examples, more applications, and more connections to 
other parts of mathematics.  

Association schemes are ubiquitous in combinatorics. For many problems the vertex set of an 
association scheme is the natural ``ambient space'' in which to search for subsets with desired structure.
Schurian association schemes are fundamental to the study of group actions. Dual polar spaces and sesquilinear
forms schemes play key roles in finite geometry. And, of course, various extremal codes and designs themselves 
form interesting association schemes.

The generic definition of a Delsarte design in an association scheme --- a subset or weighted subset of the vertex set whose characteristic  vector is orthogonal to some specified set of eigenspaces of (the Bose-Mesner algebra of) that scheme --- is again and again found to coincide with interesting and important combinatorial substructures. In \cite{martinsagan,ernstschmidt},  orthogonality to natural collections of eigenspaces translates exactly to transitivity properties of a subset of a symmetric group or finite general linear group.  In \cite{synch}, Delsarte designs play a key role in the study of synchronisation. In recent years, we have seen many  Erd\H{o}s-Ko-Rado results \cite{EKRbook} for various families of association schemes and the extremal subsets in most cases are Delsarte designs. 

It is crucial, therefore, to further develop Delsarte's theory in the case of general (commutative) association schemes and beyond. 
While the Hamming and Johnson schemes, etc., have only integer eigenvalues, most association schemes have irrational eigenvalues. The most obvious situation where restriction to a subfield of the splitting field is necessary is that in which one applies Delsarte's linear programming bound \cite[Sec.~3.2]{del} to problems in non-symmetric association schemes: here the splitting field is not contained in $\re$ so one moves to the symmetrisation to obtain a scheme with only real eigenvalues. But one can go further: since codes and designs, even weighted versions with rational weights, are represented by vectors with rational entries, we benefit by moving to rational vector spaces of matrices and, in the best scenario, work instead with a subalgebra of the Bose-Mesner
which itself is a Bose-Mesner algebra having only rational eigenvalues.

To this end, we explore the set of minimal rational idempotents in a Bose-Mesner algebra and their relation to the Galois group of the splitting field of the association scheme.  In this paper, all association schemes are commutative; so a Bose-Mesner algebra of dimension $d+1$  contains exactly $2^{d+1}$ matrix idempotents. The complex vector space spanned by any subset of these summing to $I$  is closed under both the conjugate  transpose operation and matrix multiplication.  (We will assume that our collection always includes a scalar  multiple of the all ones matrix, $J$.) A \emph{fusion scheme} arises when such a vector space is also closed under the Schur (entrywise) product. We are particularly interested in which subfields $\KK$ of the splitting field $\FF$  have the property that the matrix idempotents having entries in $\KK$ produce a fusion scheme in this way. 

The most important case for Delsarte theory, of course, is $\KK=\rats$. The matrix idempotents with rational entries are obtained by  summing (or \emph{merging}) primitive idempotents over orbits of the full Galois group $\Gal(\FF/\rats)$. These idempotents are fundamental to our understanding of $T$-designs in association schemes with irrational eigenvalues. Our main tool is based on a simple observation. Suppose $(X,\cR)$ is a $d$-class association scheme with basis $\{E_0,\ldots,E_d\}$ of primitive idempotents over its splitting field $\FF$. If a subset $C\subset X$  is a $T$-design for some subset $T$ of $\{1,\ldots,d\}$ with $j\in T$ and $E_j^{\bm{\sigma}}=E_\ell$ for $\bm{\sigma} \in \Gal(\FF/\rats)$, then $C$ is also a $\left(T \cup \{\ell\}\right)$-design. 

This leads us to study the span of all idempotent matrices in our Bose-Mesner algebra having all entries in a specified subfield $\KK$ of the splitting field $\FF$. We show, in Theorem \ref{Thm:rational}, how this can greatly reduce the size of a search space or optimisation problem. When the process of merging idempotents over orbits of $\Gal(\FF/\KK)$ results in another association scheme (see Lemma \ref{Lem:EsubK}(v) and Lemma \ref{Lem:GaloisMergings}), the Delsarte $T$-designs in the original scheme are precisely the Delsarte $T'$-design in this fusion scheme for a particular $T'$ depending on 
$T$ (Theorem \ref{Thm:inclusion}) and this has powerful implications for the collections of $T$-designs in two association schemes which happen to have isomorphic rational fusions (Theorem \ref{Thm:BijectiveDesignCollections}).

An important family of association schemes to consider are the conjugacy class schemes (or ``group schemes'') of finite groups.  These have the property that a fusion scheme is obtained by merging conjugacy classes in a way that respects the Galois group (over $\rats$). We provide a proof of this result along with some basic background material for the sake of exposition, including a natural  method for obtaining a basis of eigenvectors for the Bose-Mesner algebra.  

We finish with an illustration of how these three themes --- Delsarte theory, Galois groups, and conjugacy class schemes ---  come together. In the final section of the paper, we apply the tools of Sections  \ref{Sec:fusion}  and \ref{Sec:irrational} to look at the inner distributions and MacWilliams transforms of all subgroups of the dicyclic groups $\Dic_n$ ($n$ odd).

\section{Association schemes, Bose-Mesner algebras, and matrices over subfields} 
\label{Sec:background}

Our notation and terminology generally follow \cite{bcn} but see also \cite{banito,godsil}.
In this paper, all association schemes are commutative. An \emph{association scheme} (or, more simply, \emph{scheme}) is an ordered pair $\mathscr{X}=(X,\cR)$ where
$X$ is a nonempty finite set and $\cR=\{R_0,\ldots, R_d\}$ is a partition of $X\times X$ satisfying
\begin{itemize}
\item[(i)] $R_0= \mathrm{id}_X$, the identity relation;
\item[(ii)] for each $i$, $0\le i\le d$, there is an $i'$ for which $R_{i'}=R_i^\top = \{ (b,a) \mid (a,b)\in R_i\}$;
\item[(iii)] there exist $p_{ij}^k$ ($0\le i,j,k\le d$) such that $\left| \{ c\in X \mid (a,c)\in R_i, \ (c,b)\in R_j\} \right| =p_{ij}^k$ for any $(a,b)\in R_k$;
\item[(iv)]  $p_{ij}^k=p_{ji}^k$ for all $i$, $j$ and $k$.
\end{itemize}

We may view the pair $(X,R_i)$ as a digraph on vertex set $X$.
We write $R_i(x) = \{ y\in X | (x,y) \in R_i\}$. Condition (iii) is then $\left| R_i(a) \cap R_{j'}(b) \right| = p_{ij}^k$ whenever $b\in R_k(a)$.  We typically deal with these directed graphs via their adjacency matrices.

For a nonempty finite set $X$ and a field $\KK$, we denote by  $\Mat_X(\KK)$  the algebra of matrices with rows and columns indexed by the elements of $X$ having entries in $\KK$. For $x,y\in X$ and $A \in \Mat_X(\KK)$, the entry in row $x$, column $y$ of $A$ is denoted $A_{xy}$.  
For $[m]=\{1,\ldots,m\}$, we will simply write $\Mat_m(\KK)$ for $\Mat_{[m]}(\KK)$. For an association scheme $\mathscr{X}=(X,\cR)$ as above, the \emph{$i^{\rm th}$ adjacency matrix} or  \emph{$i^{\rm th}$ Schur idempotent}\footnote{When we use the term \emph{idempotent} without the modifier ``Schur'', we always mean a matrix $E$ such that $E^2=E$. A Schur idempotent is simply a matrix with all entries zero or one and is idempotent with respect to Schur, or entrywise, multiplication of matrices.} $A_i$ is the matrix in $\Mat_X(\rats)$ with
 $(a,b)$-entry equal to one if $(a,b)\in R_i$ and a zero otherwise. The fact that the $R_i$ partition $X\times X$ means $\sum_i A_i =J$, the all ones matrix,   $A_i\circ A_j = 0$ for $i\neq j$, and $A_i\circ A_i =A_i \neq 0$.   Conditions (i)--(iv) are then rephrased as 
\begin{itemize}
\item[(i')] $A_0= I$;
\item[(ii')] for each $i$, $A_i^\top \in \{A_0,\ldots,A_d\}$;
\item[(iii')] $A_i A_j = \sum_{k=0}^d p_{ij}^k A_k$;
\item[(iv')]  $A_i A_j = A_j A_i$.
\end{itemize}

We use $\cR=\{R_0,\ldots,R_d\}$ and $\cA=\{A_0,\ldots,A_d\}$ interchangeably.
Note that,  in this paper, a \emph{Bose-Mesner algebra} $\BMA= \cx[\cA] = \spn_\cx\{A_0,\ldots,A_d\}$ is always defined over $\cx$, by extension of scalars, if necessary.

The commuting normal matrices $A_0,\ldots,A_d$ are simultaneously diagonalisable over $\cx$ and we have a second canonical \emph{basis of primitive idempotents} $\{ E_0,E_1,\ldots,E_d\}$ for $\BMA$ satisfying $E_i E_j = \delta_{ij} E_i$ and $\sum_j E_j = I$ (see, e.g., \cite[Lem.~2.18]{BBIT}). By convention, $E_0=\frac{1}{|X|} J$. For $0\le j\le d$, there is a $j^*$ such that $E_{j^*} = E_j^\dagger$, the Hermitian transpose of $E_j$. The \emph{splitting field} of $\mathscr{X}$ (or of $\BMA$) is the subfield of $\cx$ obtained from $\rats$ by adjoining all eigenvalues of $A_1,\ldots,A_d$.

For a field $\KK\subseteq \cx$ and a set of matrices $\cE \subset \Mat_X(\KK)$, we will write $\KK[\cE]= \spn_\KK(\cE)$ the $\KK$-\emph{span} of  $\cE$.  
Assuming $\cA$ is the basis of Schur idempotents of some association scheme, $\KK[\cA]$ is always closed and commutative under both ordinary and Schur multiplication. Under these conditions, $\KK[\cA]$ is also closed under taking transpose and contains both $I$ and $J$. While the set $\KK[\cA]$ is diagonalisable over $\mathbb{C}$, it is not diagonalisable over $\KK$ in general. Among the $2^{d+1}$  idempotents in $\cx[\cA]$, we consider those with entries in subfield $\KK$: 
\begin{equation}
\label{Eqn:EK}
\cE_\KK = \left\{ F=\sum_{j=0}^d c_j E_j \middle| c_0,\ldots,c_d\in \{0,1\}, \ F\in \Mat_X(\KK) \right\}.
\end{equation}

\begin{example}
\label{Ex:Z12}
    Consider the conjugacy class scheme of $\ints_{12}$ with splitting field $\rats(\sqrt{3},i)$. Let $\zeta=e^{\pi i/6}$ and index the primitive idempotents $E_0,\ldots,E_{11}$ so that 
    $E_j = \frac{1}{12} u_j u_j^\dagger$ where $\left( u_j \right)_\ell = \zeta^{j\ell}$. For $\KK=\rats(\sqrt{-3})$, 
    $\cE_\KK = \left\{ \sum_{j=0}^{11} c_j E_j \middle| c_j \in \{0,1\}, \ c_1=c_7, \ c_3=c_9, \  c_5=c_{11} \right\}$
    and similar expressions can be derived for the subfields $\rats$, $\rats(i)$, and $\rats(\sqrt{3})$.
\end{example}

\begin{lemma}
\label{Lem:EsubK}
Let $\mathscr{X}=(X,\cA)$ be an association scheme with Bose-Mesner algebra $\BMA=\cx[\cA]$ having basis $\{E_0,\ldots,E_d\}$ of primitive idempotents. Let $\rats \subseteq \KK,\EE \subseteq \cx$. Then
\begin{itemize}
    \item[(i)] the vector space $\KK[\cE_\KK]$ 
    is closed under both conjugate transpose and ordinary matrix multiplication and contains $I$ and $J$;
    \item[(ii)] if $\KK\subseteq \EE$, then $\cE_\KK\subseteq \cE_\EE$ and $\cx\left[\cE_\KK\right] \subseteq \cx \left[\cE_\EE\right]$;
    \item[(iii)] $\cE_\EE \cap \cE_\KK = \cE_{\EE\cap \KK}$;
    \item[(iv)]  $\cE_\KK\ = \cE_\cx$ if and only if $\KK$ contains the splitting field of $\mathscr{X}$;
        \item[(v)] the complexification $\cx[\cE_{\KK}]$
    is the Bose-Mesner algebra of an association scheme if and only if $\KK[\cE_\KK]$ is closed under Schur multiplication.
\end{itemize}
\end{lemma}

\begin{proof}
    \textsl{(i)} Each $E\in \cE_\KK$ is Hermitian and those satisfying $EF=c_F E$ (some $c_F$) for all $F\in \cE_{\KK}$ form a basis of pairwise orthogonal idempotents. So we have closure under ordinary multiplication. \\
    \textsl{(ii)}, \textsl{(iii)} and \textsl{(iv)} are straightforward. \\
    \textsl{(v)} By the Bose-Mesner Theorem (\cite[Prop.~2.15]{BBIT}, \cite[Thm.~2.6.1]{bcn}), 
    a vector space $\mathbf{A}\subseteq \Mat_X(\cx)$ is the Bose-Mesner algebra of an  association scheme if and only if it is closed under transposition, matrix multiplication, and Schur multiplication and contains $I$ and $J$.  Since $\KK[\cE_{\KK}]$ is closed under the conjugate transpose operation, its Schur closure is closed under taking transposes. 
\end{proof}

We now introduce terminology for the case where condition \textsl{(v)} holds.

\begin{definition}
    Let $\mathscr{X}=(X,\cR)$ be an association scheme with Bose-Mesner algebra $\BMA$ having primitive idempotents $E_0,\ldots,E_d$. For a subfield $\KK \subseteq \cx$, we say $\mathscr{X}$ (likewise $\BMA$) has \emph{Property} $\PropM{\KK}$ if $\KK[\cE_\KK]$, where $\cE_\KK$ is 
    as defined in (\ref{Eqn:EK}), is closed under Schur multiplication. This occurs precisely when condition \textsl{(v)} of Lemma \ref{Lem:EsubK} holds.  We  denote the Bose-Mesner algebra $\cx[\cE_{\KK}]$ by  $\BMA_{\downarrow_{\KK}}$. (Note that we use the notation $\BMA_{\downarrow_{{\KK}}}$ only when $\BMA$ has Property $\PropM{\KK}$.)
\end{definition}

The two standard bases  $\{A_0,\ldots,A_d\}$ and $\{E_0,\ldots,E_d\}$ for the Bose-Mesner algebra $\BMA$ of $\mathscr{X}$  are related by the first and second eigenmatrices\footnote{We adopt the notation of \cite[Sec.~2.2]{bcn}, but see also \cite[Sec.~2.3]{del}, \cite[(3.9)]{banito},  \cite[Sec.~12.2]{godsil}, \cite[(2.4),(2.5)]{BBIT}.} $P$ and $Q$ by
\begin{equation}
\label{Eqn:PQ} 
\hfill  A_i = \sum_{j=0}^d P_{ji} E_j, \hspace{1in} E_j = \frac{1}{|X|} \sum_{i=0}^d Q_{ij} A_i.  \hfill
\end{equation}
So the matrices $P=\left[ P_{ji} \right]_{j,i=0}^d$ and $Q=\left[ Q_{ij} \right]_{i,j=0}^d$ are, up to scalar multiple, inverses of one another.  But each may be obtained from the other in another way.  Each digraph $(X,R_i)$ is regular, $A_i \ones = v_i \ones$ for $v_i=P_{0i}$. Set  $\Delta_v$ as the $(d+1)\times (d+1)$ diagonal matrix with $(i,i)$-entry $v_i$ and $\Delta_m$  the $(d+1)\times (d+1)$ diagonal matrix with $(j,j)$-entry $m_j=\rk E_j$. Then $P$ and $Q$ 
satisfy both the first and second \emph{orthogonality relations} \cite[Thm.\ II.3.5]{banito}, \cite[Sec.~2.2]{bcn} for association schemes, 
\begin{equation}
\label{Eqn:orthog} 
 \hfill PQ = |X| I , \hspace{1in}  \Delta_{m} P \Delta_m^{-1} = Q^\dagger.  \hfill 
\end{equation}
The \emph{Krein parameters} of $X$ are those scalars $q_{ij}^k$ ($0\le i,j,k\le d$) for which $E_i \circ E_j = \frac{1}{|X|}  \sum_{k=0}^d q_{ij}^k E_k$ and these also satisfy $Q_{hi}Q_{hj} = \sum_{k=0}^d q_{ij}^k Q_{hk}$ for each $0\le h\le d$. The  \emph{Krein conditions} (e.g., \cite[Thm.\ 2.3.2]{bcn}) state that, for each $i,j,k$, $q_{ij}^k \ge 0$.  We will refer to the  extension of $\rats$ by all $q_{ij}^k$ as the  \emph{Krein field}  of $\mathscr{X}$. In some cases, this is a proper subfield of the splitting field; in particular, it is contained in $\re$ even when some $A_i$ is not symmetric. This holds because each $E_i$ is a Hermitian matrix as is $E_i \circ E_j$.

\section{Fusion schemes and the Galois group}
\label{Sec:fusion}

Let $X$ be a finite nonempty set and let $\mathscr{X}=(X,\{R_0,R_1,\ldots,R_d\})$ and $\mathscr{F}=(X,\{R'_0,R'_1,\ldots$, $R'_e\})$ be association schemes on $X$. We say $\mathscr{F}$ is a \emph{fusion (scheme)} of $\mathscr{X}$ if each $R_i$ is contained in some $R'_j$; that is,
each relation $R'_j$ is expressible as a union of basis relations (or by \emph{fusing} or \emph{merging} basis relations) in $\{R_0,\ldots,R_d\}$.  Some authors also refer to $\mathscr{X}$ as a \emph{fission scheme} with respect to $\mathscr{F}$ while Godsil \cite[Chap.~5]{godsilschemes}, for example,  uses the term ``subscheme'' for what we call a fusion scheme to highlight the relationship between the corresponding Bose-Mesner algebras.  

\begin{example}
Fusion schemes play key roles in several examples of and constructions for association schemes. Several of the examples we list here are fusions of product schemes. For $\mathscr{X}=(X,\cA)$ and 
$\mathscr{Y}=(Y,\cB)$, the \emph{product scheme} $\mathscr{X}\times \mathscr{Y}$ has vertex set
$X\times Y$ and adjacency matrices $\{ A\otimes B \mid A\in \cA, \, B\in \cB\}$ and this obviously generalises to any number of components.
\begin{itemize}
\item 
If $\Gamma$ and $\Delta$ are distance-regular graphs whose cartesian product $\Gamma \, \square \, \Delta$ is also distance-regular then each of these gives us an association scheme whose relations are determined by distance in the respective graphs. In this setting, the association scheme generated in this manner by $\Gamma \, \square \, \Delta$ is a fusion of the product of the two association schemes generated by $\Gamma$ and $\Delta$. 
(See \cite{baileycameron} for other interesting fusions of product schemes.)
\item
Some distance-regular graphs are obtained by fusing classes in the association schemes
of other distance-regular graphs \cite[Sec.~4.2F]{bcn}. For example, the Ustimenko graphs \cite[Sec.~9.1]{bcn} can be constructed by fusing distances one and two in the symplectic forms dual polar space graphs.
\item
The $m$-fold extension of an association scheme $\mathscr{X}$ \cite[Sec.~2.5]{del} is a fusion of 
the $m$-fold  product of $\mathscr{X}$. 
\item
Every imprimitive association scheme contains a fusion scheme 
which restricts to a trivial scheme on each cell of the imprimitivity partition (this is a product of the corresponding quotient scheme with a trivial scheme) \cite[Sec.~2.4]{bcn} (see also \cite[Sec.~2.3]{DMM} where the term ``Bose-Mesner subalgebra'' is used).
\item
Every translation association scheme \cite[Sec.~2.10]{bcn} is a fusion scheme of the conjugacy class 
scheme (with all valencies one) of the underlying abelian group \cite[p52]{bcn}.
\item an \emph{amorphic association scheme} is an association scheme enjoying the property that every possible merging of relations yields a fusion scheme. Provided $d>2$, each basis relation is the adjacency relation of some strongly regular graph. 
\end{itemize}
\end{example}

The fusion schemes of a given association scheme form a meet semilattice. We now list this and some other basic facts about fusions.

\begin{proposition}
Let $\mathscr{X}=(X,\{R_0,\ldots,R_d\})$ be an association scheme with Bose-Mesner algebra $\BMA$  and let $\mathscr{F}_1$, $\mathscr{F}_2$ denote fusion schemes of $\mathscr{X}$ with Bose-Mesner algebras $\BMB_1$ and $\BMB_2$. 
\begin{itemize}
\item[(i)] $\mathscr{X}$ has two trivial fusions, $\mathscr{X}$  and its minimal fusion, the complete graph $(X,\{R_0, R_1\cup \cdots \cup R_d\})$;
\item[(ii)] any fusion scheme of $\mathscr{F}_1$ is also a fusion scheme of $\mathscr{X}$;
\item[(iii)] if $\mathscr{F}_1$ and $\mathscr{F}_2$ are both fusions of $\mathscr{X}$, then $\BMB_1 \cap \BMB_2$ is the Bose-Mesner algebra of some fusion scheme of $\mathscr{X}$, the unique maximal common fusion
of $\mathscr{F}_1$ and $\mathscr{F}_2$;
\item[(iv)] the splitting field of $\mathscr{F}_1$ is a subfield of the splitting field of $\mathscr{X}$. \hfill $\Box$
\end{itemize}
\end{proposition}

Given an association scheme $\mathscr{X}=(X,\cA)$ with Bose-Mesner algebra $\BMA$ and primitive idempotents $\{E_0,\ldots,E_d\}$, any partition  $\{\cO_0,\cO_1,\ldots,\cO_e\}$  of $\{0,1,\ldots,d\}$ gives
us a Schur-closed vector subspace of $\BMA$ by merging classes, namely  $\cx\left[ \left\{ \sum_{i\in \cO_j} A_i \middle|  0\le j\le e\right\}\right]$. Dually, the subspace $\cx\left[ \left\{ \sum_{i\in \cO_j} E_i \middle|  0\le j\le e\right\}\right]$ is always closed under ordinary matrix multiplication. 
We next discuss how to check whether one of these subspaces is a Bose-Mesner algebra. 

The \emph{Bannai-Muzychuk criterion} is the fundamental computational tool in testing for fusion schemes of an association scheme with known parameters. It comes in two forms.

\begin{proposition}[Bannai-Muzychuk Criterion \cite{bannai,muzy}]\
\label{Prop:BannaiMuzychuk}
Let $\mathscr{X}=(X,\{R_0,\ldots,R_d\})$ be an association scheme with eigenmatrices $P$ and $Q$.
 Let $\{\cO_0,\cO_1,\ldots,\cO_e\}$ be a partition of $\{0,1,\ldots,d\}$ with $\cO_0=\{0\}$ and define the $(d+1)\times (e+1)$ partition matrix $O$ by $O_{ij}=1$ if $i \in \cO_j$ and $O_{ij}=0$ otherwise. 
\begin{itemize}
\item[(i)] The matrix $PO$ has at least $e+1$ distinct rows; equality holds if and only if the set
$\left\{ \cup_{i\in \cO_j} R_i \middle|  0\le j\le e\right\}$ is the set of relations of a fusion scheme of $\mathscr{X}$;
\item[(ii)] The matrix $QO$ has at least $e+1$ distinct rows; equality holds if and only if the set
$\left\{ \sum_{i\in \cO_j} E_i \middle|  0\le j\le e\right\}$ is the set of minimal idempotents of
the Bose-Mesner algebra of a fusion scheme of $\mathscr{X}$. \hfill $\Box$
\end{itemize}
\end{proposition}

Suppose $\mathscr{F}=(X,\{A_0',\ldots,A_e'\})$ is a fusion scheme of $\mathscr{X}=(X,\{A_0,\ldots,A_d\})$ with $A'_j = \sum_{i\in \cO_j} A_i$ ($0\le j\le e$) as in the proposition above. The partition $\{\cO_0,\ldots,\cO_e\}$ of $\{0,1,\ldots,d\}$ induces a partition $\{ \cS_0,\ldots,\cS_e\}$ of the rows of $P$ as follows \cite[Sec.~5.2]{godsil}: the $(e+1)\times (d+1)$ matrix $S$ with $S_{ij}=1$ if $j\in \cS_i$ and $S_{ij}=0$ otherwise satisfies $PO=S P_{\mathscr{F}}$ for some $(e+1)\times (e+1)$ matrix $P_{\mathscr{F}}$ which, according to the theorem of Bannai and Muzychuk, is the first eigenmatrix of the corresponding fusion scheme. As $QP=|X|I$, 
we then have $O=\frac{1}{|X|}QS P_{\mathscr{F}}$, or $QS=OQ_{\mathscr{F}}$ where $Q_{\mathscr{F}}$ is the second eigenmatrix of $\mathscr{F}$.

 \bigskip
 
Our primary interest in fusion schemes of $\mathscr{X}$  is their connection to subfields of the
splitting field of $\mathscr{X}$.

\begin{example}
\label{Ex:Cn}
Consider the association scheme $C_n=(\ints_n,\cR)$  of the $n$-cycle, a commutative $d$-class scheme for $d=\lfloor n/2\rfloor$ with vertex set $\ints_n$. We have relations $\cR=\{R_0,\ldots, R_d\}$ defined by $(x,y)\in R_i$ if $y=x\pm i$ in $\ints_n$.  Set $R'_0=R_0$ and, for each divisor $e$ of $n$, define $R'_e = {\bigcup_{\gcd(i,n)=e} R_i}$. Then we obtain a fusion scheme $\left(\ints_n,\left\{ R'_0,\ldots,R_{\nu(n)} \right\} \right)$ of the scheme of the $n$-cycle where $\nu(n)$ is the number of positive divisors of $n$.  We will see later that all eigenvalues of this ``gcd scheme'' are rational.  All fusion schemes of this latter scheme were classified by Muzychuk \cite{muzy}.
\end{example}

\subsection{The Galois group of an association scheme}
\label{Subsec:galois}

To our knowledge, the Galois group of (the splitting field of) an association scheme $\mathscr{X}=(X,\cA)$ was first studied by Munemasa in \cite{munemasa}. We defined this above as the extension of the rationals by all eigenvalues of the scheme; alternatively, if $Q$ is the second eigenmatrix of $\mathscr{X}$, the \emph{splitting field} is defined as $\rats(Q_{11},\ldots,Q_{dd})$ and will be henceforth denoted by $\FF$.
Munemasa proved that, for any subfield $\KK$ of $\FF$ containing all the Krein parameters of $\mathscr{X}$, $\Gal(\FF/\KK)$ is contained in the center of $\Gal(\FF/\rats)$.

As each $A_i$ is rational, each $\bm{\sigma}$ in the Galois group $\Gal(\FF/\rats)$ maps each $A_i$ to itself. On the other hand, the Galois group acts faithfully on the set $\{E_0,E_1,\ldots,E_d\}$ \cite[Thm.~2.1]{finitem1}. 

\begin{definition}
Let $(X,\cA)$ be an association scheme with splitting field $\FF$ and primitive idempotents $E_0,\ldots,E_d$ and let $\rats \subseteq \KK \subseteq \FF$.
 Write $M^{\bm{\sigma}}$ for the matrix obtained by applying 
$\bm{\sigma}$ to $M$ entrywise.  For each $\bm{\sigma} \in \Gal(\FF/\KK)$, there is a  permutation $\sigma \in \Sym(\{0,\ldots,d\})$ defined by $E_{j^\sigma}=\left( E_j\right)^{\bm{\sigma}}$. We define the corresponding permutation group
\[
\Sigma_\KK:=\{ \sigma \mid \bm{\sigma} \in \Gal(\FF/\KK)\}.
\]   
\end{definition}

Since the action of $\Sigma_\rats$ is faithful,  so too is the action of $\Sigma_\KK$, since $\Gal(\FF/\KK) \leq \Gal(\FF/\rats)$ for any $\rats \subseteq \KK \subseteq \FF$. One consequence of this is that, for distinct subfields $\EE$ and $\mathbb{L}$ of the splitting field $\FF$, $\cE_\KK \neq \cE_\mathbb{L}$.
In the case where all Krein parameters are rational, we also have an action on the columns of matrix $P$ \cite[Thm.~II.7.3]{banito} by the subgroup of $\Gal(\FF/\rats)$ centralising complex conjugation. Note that, for computational purposes, one uses the fact that $\Sigma_\KK$ is  
also the permutation action of $\Gal(\FF/\KK)$ on the columns of the second eigenmatrix $Q$.

\begin{lemma}\label{Lem:GaloisMergings}
Let $\cQ_0=\{0\},\cQ_1,\ldots,\cQ_e$ be the orbits of the action of $\Sigma_\KK$ on $\{0,1,\ldots,d\}$ and, for $0\le \ell \le e$, define
\[ 
F_\ell = \sum_{j\in \cQ_\ell} E_j .
\]
Then the minimal elements of $\cE_\KK$ are precisely  $\{F_0, F_1, \ldots, F_e\}$; that is, $E\in \cE_\KK$ if and only if $E$ is expressible as a sum of some subset of $\{F_0, F_1, \ldots, F_e\}$.
\end{lemma}

\begin{proof}
It is clear that each $F_\ell$ has all entries in $\KK$ as $F_\ell^{\bm{\sigma}}=F_\ell$ for each $\bm{\sigma}\in \Gal(\FF/\KK)$. Conversely, if $E\in \cE_\KK$ with $E=\sum_{j=0}^d c_j E_j$ with 
each $c_j\in \{0,1\}$, then $c_{j^\sigma}=c_j$ for each $j$ and each $\sigma\in \Sigma_\KK$ so that $EF_\ell = c_j F_\ell$ for any $j\in \cQ_\ell$. 
\end{proof}

Recall that $\mathscr{X}$ is said to have Property $\PropM{\KK}$ if $\cE_\KK$ is closed under Schur multiplication. That is, Property $\PropM{\KK}$ holds for $\mathscr{X}$  when $\cx[\cE_{\KK}]$ is a Bose-Mesner algebra, denoted by  $\BMA_{\downarrow_\KK}$. This occurs precisely when there is a corresponding fusion of $\mathscr{X}$. In particular, this fusion can be determined by applying the Bannai-Muzychuk criterion (Proposition \ref{Prop:BannaiMuzychuk}) with the partition determined by the orbits $\cQ_\ell$ of $\Sigma_\KK$ 
on the minimal idempotents defined in Lemma \ref{Lem:GaloisMergings}. (Indeed, if $QO$ has exactly $r$ distinct rows, then the Schur closure of $\cx\left[ \cE_\KK \right]$ has dimension $r$.)

\begin{definition}
    Let $\mathscr{X}=(X,\cR)$ be an association scheme satisfying Property $\PropM{\KK}$. We denote by $\mathscr{X}_{\downarrow_{{\KK}}}$ the fusion scheme which has Bose-Mesner algebra $\BMA_{\downarrow_{{\KK}}}$. We shall call $\mathscr{X}_{\downarrow_{{\KK}}}$ the \emph{Galois fusion of $\mathscr{X}$ with respect to $\KK$} in light of Lemma \ref{Lem:GaloisMergings} and the Bannai-Muzychuk criterion. We shall refer to $\mathscr{X}_{\downarrow_{{\rats}}}$ simply as the \emph{Galois fusion} of $\mathscr{X}$.
    (Note that we use the notation $\mathscr{X}_{\downarrow_{{\KK}}}$ only when $\mathscr{X}$ has Property $\PropM{\KK}$.)
\end{definition}

The following is also well known.

\begin{proposition}
\label{Prop:symmetrization}
A commutative association scheme has only real eigenvalues if and only if it is symmetric. Given any non-symmetric commutative scheme $(X,\cR)$ with splitting field $\FF$, the Galois group $\Gal\left(\FF/(\FF \cap \re) \right)$ has complex conjugation as its sole non-identity element and the relative Galois fusion of our non-symmetric scheme with respect to  this subgroup of its Galois group is simply the symmetrisation of $(X,\cR)$.  A Galois fusion of $(X,\cR)$ must be a fusion scheme of this symmetrisation. \hfill $\Box$
\end{proposition}

\begin{remark}
    This is a special case of a much more general phenomenon: stratifiability of a homogeneous coherent configuration.
\end{remark}

\begin{proposition}
\label{Prop:MKintersection}
Let $\KK_1$ and $\KK_2$ be subfields of the splitting field of $\mathscr{X}$. If $\mathscr{X}$ satisfies Property
$\PropM{\KK_1}$ and Property $\PropM{\KK_2}$ then $\mathscr{X}$ satisfies Property $\PropM{{\KK_1 \cap \KK_2}}$.
\end{proposition}

\begin{proof}
    By Lemma \ref{Lem:EsubK}(iii), $\cx \left[ \cE_{\KK_1 \cap \KK_2} \right] =\cx \left[ \cE_{\KK_1} \cap \cE_{\KK_2} \right] =\cx \left[ \cE_{\KK_1} \right] \cap \cx \left[ \cE_{\KK_2} \right]$ is a Bose-Mesner algebra. 
\end{proof}

We now introduce some additional notation which will be useful in discussing Galois fusions and will also be used in the following sections. 

\begin{definition}
    Let $\cQ_0=\{0\},\cQ_1,\ldots,\cQ_e$ be the orbits of the action of $\Sigma_\KK$ on $\{0,1,\ldots,d\}$. We define a map $\iota:\{0, \ldots, d\} \to \{0, \ldots, e\}$ by $\iota(i) = j$ for $i \in \mathcal{Q}_j$. Given this partition of $\{0,\ldots,d\}$, we shall set $\overline{Q}=QO$ where $O$ is the $(d+1)\times (e+1)$ matrix with $(i,j)$-entry equal to one if $i\in \cQ_j$ and zero otherwise. Note that we will use $\iota_{\mathscr{X}}$ and $\overline{Q}_{\mathscr{X}}$ when such clarification is needed.
\end{definition}

The map $\iota$ is clearly surjective with $\iota^{-1}(j) = \mathcal{Q}_j$. It is injective precisely when $\FF=\KK$. Henceforth, we shall restrict $F_j$ to not simply any sums of the minimal idempotents, but to those arising from the Galois orbits with respect to some subfield $\KK$; henceforth
\[
F_j = \sum_{i \in \iota^{-1}(j)}E_i ~. 
\]

\begin{example}
\label{Ex:8vertex1}
Consider the scheme $\mathscr{X}$ built from Cayley graphs on $\ints_4 \times \ints_2$  specified by its relation matrix $\sum i A_i$ and the relation matrix of its Galois fusion $\mathscr{X}_{\downarrow_\mathbb{Q}}$ together with their corresponding second eigenmatrices $Q_\mathscr{X}$ and $Q_\mathscr{\mathscr{\mathscr{X}_{\downarrow_\mathbb{Q}}}}$.
Here we see this last matrix is obtained from $\overline{Q}_\mathscr{X}$ by removal of repeated rows.
\[
\mathscr{X}:
\begin{bmatrix}
  0 &  1 &  2 &  3 &  4 &  4 &  4 &  4 \\
  1 &  0 &  3 &  2 &  4 &  4 &  4 &  4 \\
  3 &  2 &  0 &  1 &  4 &  4 &  4 &  4 \\
  2 &  3 &  1 &  0 &  4 &  4 &  4 &  4 \\
  4 &  4 &  4 &  4 &  0 &  1 &  2 &  3 \\
  4 &  4 &  4 &  4 &  1 &  0 &  3 &  2 \\
  4 &  4 &  4 &  4 &  3 &  2 &  0 &  1 \\
  4 &  4 &  4 &  4 &  2 &  3 &  1 &  0 
\end{bmatrix},
\quad
\mathscr{\mathscr{X}_{\downarrow\mathbb{Q}}}:
\begin{bmatrix}
   0 &  1 &  2 &  2 &  3 &  3 &  3 &  3 \\
    1 &  0 &  2 &  2 &  3 &  3 &  3 &  3 \\
    2 &  2 &  0 &  1 &  3 &  3 &  3 &  3 \\
    2 &  2 &  1 &  0 &  3 &  3 &  3 &  3 \\
    3 &  3 &  3 &  3 &  0 &  1 &  2 &  2 \\
    3 &  3 &  3 &  3 &  1 &  0 &  2 &  2 \\
    3 &  3 &  3 &  3 &  2 &  2 &  0 &  1 \\
    3 &  3 &  3 &  3 &  2 &  2 &  1 &  0 
\end{bmatrix}
\]
\[
Q_\mathscr{X}=
\left[ \begin{array}{crrrr}
        1 &        1 &        2 &        2 &        2 \\
          1 &        1 &       -2 &       -2 &        2 \\
          1 &        1 &  -2i &   2i &       -2 \\
          1 &        1 &   2i &  -2i &       -2 \\
          1 &       \!\!-1 &        0 &        0 &        0 
\end{array} \right],
\ 
\overline{Q}_\mathscr{X}= \left[
\begin{array}{rrrr}
        1 &        1 &        4 &               2 \\
        1 &        1 &       \!-4 &               2 \\
        1 &        1 &  0 &        \!-2 \\
        1 &        1 &   0 &         \!-2 \\
        1 &       \!-1 &              0 &        0 
\end{array} \right],
\ 
Q_\mathscr{\mathscr{\mathscr{X}_{\downarrow_\mathbb{Q}}}}= \left[
\begin{array}{rrrr}
    1&   1&   2&   4 \\
     1&  \!\!-1&   0&   0 \\
     1&   1&  \!\!-2&   0 \\
     1&   1&   2&  \!\!-4 
  \end{array}\right]
\]
The splitting field of $\mathscr{X}$ is $\mathbb{Q}(i)$ and the permutation action of $\Gal(\mathbb{Q}(i)/\mathbb{Q})$ on the idempotents gives $\langle (2,3) \rangle$. Hence $F_0 = E_0$, $F_1 = E_1$, $F_2 = E_2 + E_3$, $F_3 = E_4$ and $\iota: \{0, 1, 2, 3, 4\} \to \{0, 1, 2, 3\}$ by $\iota(0) = 0$, $\iota(1) = 1$, $\iota(2) = 2$, $\iota(3) = 2$ and $\iota(4) = 3$. This is described by $\overline{Q}$ where columns 2 and 3 of $Q$ have been replaced by their sum. Repeated rows indicate that relations 2 and 3 fuse, and by the Bannai-Muzychuk Criterion, we have a valid fusion $\mathscr{\mathscr{X}_{\downarrow_\mathbb{Q}}}$ with 2nd eigenmatrix $Q_\mathscr{\mathscr{\mathscr{X}_{\downarrow_\mathbb{Q}}}}$.
\end{example}

Note that not every scheme has property $\PropM{\rats}$ as the following example shows.

\begin{example}
\label{Ex:CoxeterFails}
The Coxeter graph (considered as a metric association scheme) has second eigenmatrix $Q$, below, 
and corresponding matrix $\overline{Q}$:
    \[
Q =
\begin{bmatrix}
  1 & 8 & 6 & 7 & 6 \\
    1 & 16/3 & -2+2\sqrt{2} & -7/3 & -2-2\sqrt{2} \\
    1 &  4/3 & -2\sqrt{2} & -7/3 & 2\sqrt{2} \\
    1 & -4/3 & -1 &  7/3 &  -1\\
    1 & -8/3 & 2+\sqrt{2} & -7/3 & 2-\sqrt{2} 
  \end{bmatrix},
  \quad
 \overline{Q} =
\begin{bmatrix}
  1 & 8 & 12 & 7 \\
    1 & 16/3 & -4 & -7/3  \\
    1 & \phantom{-}4/3 & \phantom{-}0 & -7/3 \\
    1 & -4/3 & -2 &  \phantom{-}7/3 \\
    1 & -8/3 & \phantom{-}4 & -7/3 
  \end{bmatrix}
\]
That is, we have summed the columns of $Q$ according to the Galois orbits to obtain $\overline{Q}$ which clearly fails the Bannai-Muzychuk criterion. Hence the Coxeter graph does not satisfy property $\PropM{\rats}$. This means that the Coxeter graph does not have a  Galois fusion over any proper subfield.
\end{example}

\begin{problem}\label{Prob:Property}
For any association scheme $\mathscr{X}$, when does Property $\PropM{\KK}$ hold for $\rats \subseteq \KK \subseteq \FF$? In particular, when does Property $\PropM{\rats}$ hold? That is, what are the Galois fusions of $\mathscr{X}$ with respect to subfields of $\FF$?
\end{problem}

\section{Delsarte designs and irrational eigenvalues}
\label{Sec:irrational}

We continue with the notation used in Section \ref{Sec:background} and Section \ref{Sec:fusion}.

For $T\subseteq \{1,\ldots,d\}$ a subset $\emptyset \subset C \subseteq X$ is a \emph{Delsarte $T$-design} if  the characteristic vector $\bx=\bx_C$ of $C$ (with $\bx_a=1$ if $a\in C$ and $\bx_a=0$ otherwise) satisfies $E_j \bx = 0$ for all $j\in T$. Given $C$, the largest set  of indices for which $C$ is a design is $T(C)=\{ j \mid E_j\bx=0\}$.  The \emph{inner distribution} $a=a(C)=[a_0,\ldots,a_d]$ and \emph{dual distribution} (or \emph{MacWilliams transform}) $b=b(C)=[b_0,\ldots,b_d]$ defined by
$$\hfill  a_i = \frac{1}{|C|} \bx^\top A_i \bx  \hspace{1in}  b_k = \frac{|X|}{|C|} \bx^\top E_j \bx \hfill $$
are related by $b=aQ$ and we know $a\ge 0$, $b\ge 0$ (see, e.g., \cite[Sec.~2.5]{bcn}).  Evidently, $j\in T(C) \Leftrightarrow b_j=0$.

We may also consider weighted Delsarte $T$-designs where each vertex in $C$ is assigned a weight in $\rats$. If the weights are in $\ints_+$ then they may be considered as ``repeated points'' or as a multiset. Any rationally weighted Delsarte T-design corresponds to a vector $\bx$ with rational entries. As for non-weighted Delsarte $T$-designs, we have the same properties for $T(C)$ and the MacWilliams transform, after suitably adjusting the inner distribution to
\[
a_i = \frac{\bx^\top A_i \bx}{\bx^\top \bx}.
\]

We are typically interested primarily in 01-vectors characteristic vectors. However, even when considering weighted Delsarte designs the characteristic vector has rational entries. This means that we have $\bx_C = \sum_{i \not \in T(C)} E_j \bx_C$ in $\rats^{X}$. Hence if $F_0, \ldots, F_e$ are the minimal elements of $\mathcal{E}_\rats$, then $\bx_C$ must be expressible as
$\bx_C = \sum_{j \in S} F_j\bx_C$ 
for some $S \subseteq \{0, \ldots, e\}$. We will develop this idea in what follows, so we restrict ourselves to $\KK=\rats$. We note, however, that the results generalise to other choices of $\KK$.

\begin{theorem}
\label{Thm:rational}
Let $(X,\cR)$ be an association scheme and let $C\subseteq X$ be a nonempty (possibly weighted) subset of the vertices with inner distribution $a$ and characteristic vector $\bx$.
\begin{enumerate}
    \item[(i)] If $j\in T(C)$ and $\sigma\in \Sigma_\rats$, then $j^\sigma \in T(C)$.  Hence $T(C)$ is a union of orbits $\cQ_j \in \{ \cQ_1,\ldots, \cQ_e\}$.
    \item[(ii)] Let $\iota(i)=\iota(j)$. If $aQ_i = 0$ then $aQ_j = 0$. Equivalently, if $E_i \bx = 0$ then $E_j \bx = 0$.
    \item[(iii)] If $aQ_i = 0$ then $a\overline{Q}_{\iota(i)}=0$. Equivalently, if $E_i \bx = 0$ then $F_{\iota(i)} \bx = 0$.
\end{enumerate}
\end{theorem}

\begin{proof}
(i) For $\bm{\sigma} \in \Gal(\FF/\rats)$ and corresponding $\sigma\in \Sigma_\rats \subseteq \Sym(\{0,\ldots,d\})$, we have 
$$ (aQ)^{\bm{\sigma}}= a(Q)^{\bm{\sigma}}= [b_0,b_{1^\sigma},\ldots,b_{d^\sigma}] $$
so that $b_j=0$ if and only if $b_{j'}=0$ for all $j'$ in the orbit $\cQ_\ell$ containing $j$. Parts
(ii) and (iii) follow immediately.
\end{proof}

Theorem \ref{Thm:rational} has important implications in practice, since it means that the study of Delsarte designs in an association scheme with irrational eigenvalues is essentially equivalent to the study of 
subsets satisfying $F_j \bx=0$ for all $j$ in some reduced set $T' \subseteq \{1,\ldots,e\}$. In many cases, this greatly reduces the complexity of a search, for example when using linear programming.

\begin{example}
Recall the Coxeter graph from Example \ref{Ex:CoxeterFails}. It can be viewed as being obtained from the odd graph $O_4$ by deleting the vertices of a Fano plane. 
The Galois merging of idempotents for the Coxeter graph is given by $\iota:\{0,1,2,3,4\} \to \{0,1,2,3\}$ where $\iota(0)=0$, $\iota(1)=1$, $\iota(2)=2$, $\iota(3)=3$, $\iota(4)=2$. Hence, for example, any $\{1,2\}$-design or $\{1,4\}$-design of $\mathscr{X}$ is in fact a 
$\{1,2,4\}$-design, since $\iota(\{1,2,4\}) = \iota(\{1,2\})= \iota(\{1,4\}) = \{1,2\}$. For instance, a
 second Fano plane taken from the remaining vertices gives a $\{1, 2, 4\}$-design with inner distribuition $a = [1, 0, 0, 6, 0]$. Note that
$aQ = [7, 0, 0, 21, 0 ]$ and $a\overline{Q} =[7, 0, 0, 21]$.
\end{example}

We have a stronger result than Theorem \ref{Thm:inclusion} when property $\PropM{\mathbb{Q}}$ holds.

\begin{theorem}\label{Thm:inclusion}
Let $\mathscr{X}$ be a $d$-class scheme satisfying property $\PropM{\mathbb{Q}}$ and let $T \subset \{1, \ldots, d\}$. Then $C \subset X$ is a $T$-design of $\mathscr{X}$ if and only if it is a $\iota(T)$-design of $\mathscr{X}_{\downarrow_\mathbb{Q}}$. \hfill $\Box$
\end{theorem}

This says in essence that it is sufficient to consider the $T$-designs of the $\mathscr{X}_{\downarrow_\mathbb{Q}}$ when studying $T$-designs of $\mathscr{X}$. This is helpful in particular when showing non-existence or when classifying $T$-designs. Note however, that they may have different combinatorial properties.
This is essentially used in an ad hoc manner in \cite{synch}.

\begin{example}
\label{Ex:8vertexDesign}
Recall $\mathscr{X}$ and $\mathscr{\mathscr{X}_{\downarrow_\mathbb{Q}}}$ from Example \ref{Ex:8vertex1}. Let $C = \{  1, 2, 5, 6 \} \subset X = \{1, \ldots, 8\}$. Now $C$ is a $\{1, 2, 3\}$-design in $\mathscr{X}$ with inner distribution $[1, 1, 0, 0, 2]$. Moreover $C$ is a $\{\iota(1), \iota(2), \iota(3)\} = \{1, 2\}$-design in $\mathscr{\mathscr{X}_{\downarrow_\mathbb{Q}}}$. With respect to $\mathscr{\mathscr{X}_{\downarrow_\mathbb{Q}}}$
the inner distribution of $C$ is $[ 1, 1, 0, 2 ]$.
We note that equivalently $[1, 1, 0, 0, 2]\overline{Q}_\mathscr{\mathscr{X}} = [ 1, 1, 0, 2 ] Q_{\mathscr{X}_{\downarrow_\mathbb{Q}}}$ ~.
\end{example}

Theorem \ref{Thm:inclusion} establishes a correspondence between the designs of $\mathscr{X}$ and $\mathscr{X}_{\downarrow_\rats}$ which has curious implications for pairs of non-isomorphic schemes having isomorphic Galois fusions.
We call $\mathscr{X} =(X, \mathcal{R})$ isomorphic to $\mathscr{Y}=(Y, \mathcal{S})$, $\cS=\{S_0,\ldots,S_e\}$, 
if there is a bijection $\sigma_1$ from $X$ to $Y$ and a bijection $\sigma_2$ from 
$\{0, \ldots, d\}$ to $\{0, \ldots, e\}$ such that $\{\left(\sigma_1(x), \sigma_1(y)\right) \mid (x,y) \in R_i\} = S_{\sigma_2(i)}$ 
in which case we write $\mathscr{X}^{(\sigma_1, \sigma_2)} = \mathscr{Y}$. 
For some $T \subseteq \{1, \ldots, d\}$ we define 
$\mathfrak{D}_T(\mathscr{X}) = \{ C \subset X : C \text{ is a $T$-Design of } \mathscr{X}\}$, 
that is, $\mathfrak{D}_T(\mathscr{X})$ is the collection of all $T$-designs of $\mathscr{X}$.

\begin{theorem} 
\label{Thm:BijectiveDesignCollections}
Let $\mathscr{X}$ and $\mathscr{Y}$ respectively be $d$-class and $e$-class schemes that satisfy property $\PropM{\rats}$, and let $T \subset \{1, \ldots, d\}$ and $T' \subset \{1, \ldots, e\}$. If $\mathscr{X}_{\downarrow_\mathbb{Q}}^{(\sigma_1, \sigma_2)} = \mathscr{Y}_{\downarrow_\mathbb{Q}}$ and $\iota_\mathscr{X}(T) = \iota_\mathscr{Y}(T')$ then $C \mapsto C^{\sigma_1}$ is a bijection from $\mathfrak{D}_T(\mathscr{X})$ to $\mathfrak{D}_{T'}(\mathscr{Y})$. \hfill $\Box$
\end{theorem}

\begin{example}
Let $\mathscr{X}$ be as in Example \ref{Ex:8vertex1} and let $\mathscr{Y}$ be the association scheme on the same vertex set $\ints_4\times \ints_2$ defined by the following relation matrix $\sum i A_i$ and having second eigenmatrix $Q_\mathscr{Y}$ as follows:
\[
\mathscr{Y}:
\begin{bmatrix}
   0 &  1 &  2 &  2 &  3 &  3 &  4 &  4 \\
    1 &  0 &  2 &  2 &  3 &  3 &  4 &  4 \\
    2 &  2 &  0 &  1 &  4 &  4 &  3 &  3 \\
    2 &  2 &  1 &  0 &  4 &  4 &  3 &  3 \\
    4 &  4 &  3 &  3 &  0 &  1 &  2 &  2 \\
    4 &  4 &  3 &  3 &  1 &  0 &  2 &  2 \\
    3 &  3 &  4 &  4 &  2 &  2 &  0 &  1 \\
    3 &  3 &  4 &  4 &  2 &  2 &  1 &  0 
\end{bmatrix},
\quad
Q_\mathscr{Y}=
\left[ \begin{array}{rrrrr}
       1 &      1 &      1 &      1 &      4 \\
        1 &      1 &      1 &      1 &     \!\!-4 \\
        1 &     \!\!-1 &     \!\!-1 &      1 &      0 \\
        1 &  \!\!-i &   i &     \!\!-1 &      0 \\
        1 &   i &  \!\!-i &     \!\!-1 &      0 
\end{array} \right]
\]
Then $\mathscr{X}$ and $\mathscr{Y}$ are non-isomorphic but $\mathscr{X}_{\downarrow_\mathbb{Q}} = \mathscr{Y}_{\downarrow_\mathbb{Q}}$.
Now $\iota_\mathscr{Y}:\{0,1,2,3,4\} \to \{0,1,2,3\}$ by $\iota_\mathscr{Y}(0)=0$, $\iota_\mathscr{Y}(1)=1$, $\iota_\mathscr{Y}(2)=1$, $\iota_\mathscr{Y}(3)=2$, and $\iota_\mathscr{Y}(4)=3$. Consider $C$ from Example \ref{Ex:8vertexDesign} which is a $\{1,2,3\}$-design in $\mathscr{X}$ and a $\{1, 2\}$-design in $\mathscr{\mathscr{X}_{\downarrow_\mathbb{Q}}} = \mathscr{\mathscr{Y}_{\downarrow_\mathbb{Q}}}$. So $C$ is a $\iota_\mathscr{Y}^{-1}(1) \cup \iota_\mathscr{Y}^{-1}(2) = \{3, 4\}$-design of $\mathscr{Y}$. Note that $\mathscr{X}_{\downarrow_\mathbb{Q}}$ and $\mathscr{Y}_{\downarrow_\mathbb{Q}}$ are equal and not just isomorphic, so no permutation of the vertices need be applied to $C$ for it to be a design of $\mathscr{Y}$.
It is not hard for the interested reader to completely enumerate all $T$-designs for both $\mathscr{X}$ and $\mathscr{Y}$ by computer. Indeed, in light of Theorem \ref{Thm:BijectiveDesignCollections}, one need only enumerate them for $\mathscr{X}_{\downarrow_\mathbb{Q}}$.
\end{example}

\section{Conjugacy Class Association Schemes}
\label{Sec:conjschemes}

We begin this section with a brisk review of the general theory of conjugacy class association schemes. These are well-studied and serve as  fundamental examples in some standard texts; see 
\cite[p.~54]{banito}, 
\cite[pp.~230-1,264]{godsil}, 
\cite[p.~51]{BBIT}, but also 
\cite[pp.~19-21]{banito} and 
\cite[Sec.~2.9]{bcn}. 

Let $G$ be a finite group with decomposition $G = \cC_0\cup \cC_1 \cup \cdots \cup \cC_d$ into conjugacy classes where $\cC_0=\{1\}$. 
We build matrices in $\Mat_G(\cx)$ as follows.  For $0\le i\le d$, define $A_i$, the $i^{\rm th}$ adjacency matrix, as the 01-matrix in $\Mat_G(\cx)$ having  a one in position $(g,h)$ if $g^{-1}h\in \cC_i$. This is the adjacency matrix of the digraph with vertex set $G$ and arc set
$R_i = \left\{ (g,ga) \middle| g\in G, \ a\in \cC_i \right\}$. 
The Bose-Mesner algebra $\BMA$,  defined above as the vector space 
span of $\{A_0,\ldots,A_d\}$,  is also the center of the group algebra $\cx [\hat{G}]$ where $\hat{G}$ is the group of $v=|G|$ permutation matrices in $\Mat_G(\cx)$ given by the right regular representation $g\cdot x=xg^{-1}$; the permutation matrix $P_g \in \hat{G}$ has a one in position $(x,xg^{-1})$ for $x\in G$ and $P_g \be_x = \be_{xg}$ where $\be_{x} \in \cx^G$ is the standard basis vector corresponding to $x\in G$.

We have a partition $\cR=\{R_0,\ldots,R_d\}$ of $G\times G$ into binary relations.
It is well-known that $(G,\cR)$ is a commutative association scheme \cite[p.~20]{banito}; the scheme is symmetric precisely when all
conjugacy classes are inverse-closed. 
Each normal subgroup of $G$ determines an imprimitivity system of this association scheme; if $N \unlhd G$, then $N$ is expressible as a union of conjugacy classes, $\cC_0\cup \cdots \cup \cC_s$ say, and $(X,R_1\cup \cdots \cup R_s)$ is a disjoint union of $|G:N|$ complete graphs of size $|N|$. Moreover, every imprimitivity system of a conjugacy class scheme arises in this way: the component 
of the disconnected graph containing the identity is a union of conjugacy classes closed under multiplication.

There are exactly $d+1$ irreducible characters of $G$ including the trivial character $\chi_0:G \rightarrow \{1\}$. The \emph{character table} $\sT$ of $G$ has one row for each irreducible character and one column for each conjugacy class; 
in row $j$, column $i$, the value is $\chi_j(g)$ for $g\in \cC_i$.  We denote the basis of primitive idempotents of the 
Bose-Mesner algebra $\BMA$ by $\{E_0,\ldots,E_d\}$ and define the first and second \emph{eigenmatrices} $P$ and $Q$ by the equations (\ref{Eqn:PQ}) as usual.

Though it is well-known, we will show below using eigenvectors of the $A_i$ that  for the conjugacy class scheme of a finite group $G$, we have $Q_{ij} = \chi_j(1) \, \overline{\chi_j(g)}$ and $P_{ji} = |\cC_i|\chi_j(g)/\chi_j(1)$ ($g\in \cC_i$)\footnote{For this reason, many authors refer to $P$ as the ``character table'' of $G$ though some elementary scaling is involved in going from $\sT$ to $P$}. Denoting the corresponding irreducible representations by $\rho_0,\ldots,\rho_d$  
we have $\chi_j(g)=\tr \rho_j(g)$ and the 
right regular representation $\rho:G\rightarrow \hat{G}$ can be block diagonalised as a direct sum of irreducible representations \cite[Th.~18.10]{df}; in the notation of Bannai and Ito \cite[p.~17]{banito}, we have
$$ \rho \sim \rho_0 + f_1 \rho_1 + \cdots + f_d \rho_d $$
where, in the block diagonalisation of the right regular representation $\rho$, the $j^{\rm th}$ irreducible representation
$\rho_j$ appears with multiplicity equal to its degree $f_j$.

Our final goal in this section is to diagonalise the Bose-Mesner algebra of a conjugacy class association scheme.

For $0\le j\le d$, fix a choice of representation $\rho_j:G\rightarrow \Mat_{f_j}(\cx)$ with $\tr \rho_j(g)=\chi_j(g)$ for each $g\in G$. With $f=f_j$, consider the  bijection
$\upsilon:[f]\times [f] \rightarrow [f^2]$, vectorising each matrix $R=[r_{ij}] \in \Mat_f(\cx)$ into $[r_{11},r_{12},\ldots,r_{ff}]$. Construct a $|G|\times f^2$ matrix  $U_j$ with row $g$ equal to $u_g:=\upsilon\left( \rho(g)\right)$.

\begin{theorem}
\label{Thm:diagonalise}
The representations $\rho_0,\ldots,\rho_d$ give us $|G|$ linearly independent eigenvectors for $\BMA$. More precisely, the $\sum_{j=0}^d f_j^2$ columns of the partitioned matrix $\left[ U_0 | U_1 | \cdots | U_d \right]$ form a basis for $\cx^G$ where every basis vector is a common eigenvector for $A_0,\ldots,A_d$.
\end{theorem}

\begin{proof}
Let $j$ be given, $0\le j\le d$, and write simply $U=U_j$ as constructed above.
We will show that, for each $i$, $0\le i\le d$, 
$$A_i U = \frac{ |\cC_i| \chi_j(g_i)}{f_j} U$$
where $g_i\in \cC_i$.
Let $C = \sum_{x\in \cC_i} \rho(x)$. Then for any $g\in G$, $\rho(g)^{-1} C\rho(g) = C$ so that $C\rho(g) = \rho(g) C$. By Schur's Lemma \cite[Prop.~4, 2.2]{serre}, \cite[Thm.~1.4]{ledermann}, 
$C=\theta I_{f_j}$ 
for some $\theta$. Thus $\sum_{x\in \cC_i} \rho(gx)= \rho(g) C = \theta \rho(g)$ and row $g$ of $A_iU$ is 
\[
     \sum_{x\in \cC_i} u_{gx} =\upsilon\left( \sum_{x\in \cC_i} \rho({gx}) \right) = \theta u_g. 
\]
In other words $A_iU=\theta U$ for some $\theta \in \cx$. Taking the trace of both sides of the equation $\sum_{x\in \cC_i} \rho(x) = \theta I$, we  find $ |\cC_i| \chi_j(g_i) = \theta f_j$.   

By Wedderburn's structure theorem for semisimple algebras (see also \cite[p48]{serre}), the image $\{ \rho(g) \mid g\in G\}$ is full-dimensional in $\Mat_f(\cx)$. In fact, extending $\rho$ linearly to the full group algebra, we have that $\rho: \cx[G] \rightarrow \Mat_{f_j}(\cx)$ is surjective.
So the $f_j^2$ columns of $U$
are linearly independent. Distinct irreducible representations $\rho_j,\rho_k$ of $G$ correspond to orthogonal columns of the matrix $Q$ so these $f_j^2$ vectors are orthogonal to the eigenvectors obtained in this same manner from the remaining
$d$ irreducible representations of $G$. So we have $\sum_{j=0}^d f_j^2$ linearly independent eigenvectors for the Bose-Mesner algebra. We then use the standard identity $\sum_{j=0}^d f_j^2=|G|$  (see, e.g., \cite[Theorem 18.10(3)]{df}) to finish the proof.  
\end{proof}

The first and second eigenmatrices of the conjugacy class association scheme $P=[P_{ji}]$ and $Q=[Q_{ij}]$  defined by (\ref{Eqn:PQ}) above satisfy standard first and second orthogonality relations for association schemes (\ref{Eqn:orthog}) where $v_i=|\cC_i|$, but in this case, of course, these reduce to the classical orthogonality relations for group characters \cite[Thm.\ I.4.4, I.4.9]{banito}, \cite[Sec.~2.3]{serre}. Our construction of matrix $U$ above shows $m_j=f_j^2$. So the second orthogonality relation gives us $Q_{ij} = f_j \overline{\chi_j(x_i)}$.   In summary (cf. \cite[Thm.\ II.7.2]{banito}, with $g_i$ a representative of $\cC_i$, 
\begin{equation}
\label{Eq:groupPQ}
 P_{ji} =  \frac{1}{f_j} |\cC_i| \chi_j(g_i), \hspace{1in} Q_{ij} = f_j \ \overline{\chi_j(g_i)}.
\end{equation}

The Krein parameters of the conjugacy class scheme are closely related to the decomposition of Kronecker products of irreducible representations into irreducibles.  If $f_i$ denotes the degree of $\rho_i$ and 
\[ 
\rho_i \otimes \rho_j \sim  r_{ij}^0 \rho_0 + \cdots + r_{ij}^d \rho_d ,
\]
then $\chi_i \circ \chi_j = \sum_{k=0}^d r_{ij}^k \chi_k$ and, after scaling each column by its degree and using \cite[Lemma~2.3.1(vii)]{bcn},
we find  $q_{ij}^k = \frac{f_if_j}{f_k} r_{ij}^k$. 

The conjugacy classes of size one are precisely $\left\{ \{a\} \middle| a\in Z(G) \right\}$ and the eigenspaces of dimension one correspond to the elements of the abelianisation $G'=G/[G,G]$ of $G$.

The intersection numbers $p_{ij}^k$ are given by
$$ p_{ij}^k = \left| \cC_i \cap z\cC_j^{-1} \right|$$
where $z \in \cC_k$ and $\cC_j^{-1} = \{ h^{-1} \mid h\in \cC_j\}$. To see this, consider any $x,y\in G$ with $x^{-1}y\in \cC_k$; find $g\in G$
with $g^{-1} x^{-1} y g = z$ so that $xg^{-1} \left( \cC_i \cap R_{j'}(z) \right) g = R_i(x) \cap R_{j'}(y)$.

\subsection{Galois group and rational fusions}

From here forward let $\zeta_\ell$ denote a primitive (complex)  $\ell^{\rm th}$ root of unity. 

The splitting field of the finite group $G$ is defined to be the extension of $\rats$ by the entries of its character table. This then agrees with the concept of the splitting field \cite{munemasa} of the conjugacy class association scheme by (\ref{Eq:groupPQ}). Let $\FF$ denote the splitting field of the conjugacy class scheme of $G$. As discussed in Section \ref{Subsec:galois},  each $\bm{\sigma}\in \Gal(\FF/\rats)$ acts on the set of primitive idempotents $\{E_0,\ldots,E_d\}$; to $\bm{\sigma}$, associate
the permutation  $\sigma \in \Sigma_\rats$ 
defined by $E_{j^\sigma}=\left( E_j\right)^{\bm{\sigma}}$ where the action of $\bm{\sigma}$ on matrices is entry by entry.

Recall that $x,y\in G$ are \emph{rationally conjugate} if $g\langle x\rangle g^{-1} = \langle y \rangle$ for some $g\in G$. 
Clearly if $x$ and $y$ are conjugate in $G$, then $x$ and $y$  are rationally conjugate and any equivalence class $\cC$ of this ``rationally conjugate'' relation is a union of conjugacy classes of $G$ and so $R_{\cC}=\{ (x,y) \mid x^{-1}y\in \cC \}$ is expressible as a union of basis relations of the conjugacy class scheme.

The exponent of finite group $G$ divides $n=|G|$. If $G$ has exponent $k$, then, in any representation $\rho:G\rightarrow \GL_m(\cx)$, $\rho(g)^k=I$ for each $g\in G$ so the eigenvalues $\rho(g)$ are $k^{\rm th}$ roots of unity. Hence the trace $\chi(g) =\tr \rho(g)$ belongs to the cyclotomic extension $\rats( \zeta_k)$ and this is contained in $\rats( \zeta_n)$. In some cases (e.g., whenever all conjugacy classes are inverse-closed and the eigenvalues are all real), the splitting field of $G$ can be a proper subfield of  $\rats( \zeta_k)$. Motivated by this important class of examples, Simon Norton asked whether the splitting field of any commutative association scheme is contained in some cyclotomic extension of the rational numbers \cite[p183]{banito}.

While we choose here to cite a reference dealing directly with Schur rings, the connection between rational characters and rational conjugacy classes of $G$ is well studied (indeed, this connection is the origin of the term ``rational conjugacy class''). For instance, the fact that there are equally many rational conjugacy classes of $G$ as orbits of the full Galois group can be found in Serre, \cite[Chap.~13]{serre}.

\begin{theorem}[{Muzychuk, et al.\ \cite{MPC}}]  
\label{Thm:fusion}
The conjugacy class association scheme $(G,\cR)$ of any finite group $G$ satisfies Property $\PropM{\rats}$. The basis relations of the corresponding fusion scheme are  the relations $R_{\cC}$ as $\cC$ ranges over the rational conjugacy classes of $G$ as defined above.
\end{theorem}

\begin{proof}
An idempotent $\sum_{j=0}^d c_jE_j$ belongs to $\cE_\rats$ if only if  the corresponding character $\psi(g) = \sum_{j=0}^d c_j \chi_j(g)$ is rational for each $g\in G$.

As shown above, the splitting field $\FF$ is contained in a cyclotomic extension $\rats(\zeta_n)$ and, for each $m$ with 
$\gcd(m,n)=1$, there is a Galois automorphism 
$\bm{\sigma}\in \Gal(\rats(\zeta_n)/\rats)$ with $\zeta_n^{\bm{\sigma}}=\zeta_n^m$.

Now suppose $\langle x\rangle = \langle y\rangle$. Then $y=x^m$ for some $m$ with $\gcd(m,n)=1$. If $\psi$ is a rational character of $G$, say  $\psi(g) = \tr \rho(g)$ for some complex representation $\rho:G \rightarrow \GL_f(\cx)$, then  $\psi(x) = \tr \rho(x) = \sum_{j=1}^f \omega_j$
for some  $n^{\rm th}$ roots of unity $\omega_1,\ldots,\omega_f$. We then have 
$$ \psi(y) = \tr \, \rho\left( x^m \right) = \tr \, \rho\left( x\right)^m =  \sum_{j=1}^f \omega_j^m =\sum_{j=1}^f \omega_j^{\bm{\sigma}} =\left( 
\sum_{j=1}^f \omega_j  \right)^{\bm{\sigma}} = \psi(x)^{\bm{\sigma}} = \psi(x)$$
since $\psi$ is rational. Now if $y$ is instead conjugate to some $x^m$ with $\langle x^m\rangle = \langle x \rangle$, then  $\psi(y) = \psi\left(x^m\right)=\psi(x)$.
This tells us that the adjacency matrix of the relation $R_{\cC}$,  $A=\displaystyle{\sum_{\cC_i\subseteq \cC}} A_i$, satisfies $A\circ E= \alpha_E A$ (some $\alpha_E$) for each $E\in \cE_\rats$ so $A$ belongs to the Schur closure of $\cx[\cE_\rats]$. Equivalently, collapsing over 
orbits of $\Gal(\FF/\rats)$, we find that rows $i$ and $j$ of $\overline{Q} =QO$ are equal whenever $\cC_i$ and $\cC_j$ are contained in the same rational conjugacy class.
So the number $e+1$ of Galois orbits $\cQ_j$ is at most the number of rational conjugacy classes. Going in the other direction, the adjacency matrix $A$ of any relation $R_\cC$ belongs to the full Bose-Mesner algebra and has only rational eigenvalues, so belongs to $\rats[\cE_\rats]$.  
\end{proof}

\begin{example}
    The ``gcd scheme'' for $\ints_n$ described in Example \ref{Ex:Cn} is precisely the rational fusion of the conjugacy class scheme of $\ints_n$.
\end{example}

\begin{example}
The alternating group $A_4$ has character table and eigenmatrices
$$ \sT=\left[ \begin{array}{crcc}  
1 & 1 & 1 & 1 \\
1 & 1 & \zeta\  & \zeta^2  \\
1 & 1 & \zeta^2 & \zeta\   \\
3 &\!\!-1 & 0 & 0  \end{array} \right], \quad
Q =\left[ \begin{array}{cccr}  
1 & 1 & 1 & 9 \\
1 & 1 & 1 & \!\!-3 \\
1 & \zeta & \zeta^2 & 0 \\
1 & \zeta^2 & \zeta & 0 \end{array} \right], \quad
P =\left[ \begin{array}{crcc}  
1 & 3 & 4 & 4 \\
1 & 3 & 4 \zeta & 4\zeta^2 \\
1 & 3 & 4 \zeta^2 & 4\zeta \\
1 &\!\!-1 & 0 & 0 \end{array} \right]$$
where $\zeta$ is a primitive cube root of unity. We fuse the conjugate pair of eigenspaces
to find $\bar{Q} = \left[ \begin{array}{crr}  
1 & 2 & 9 \\
1 & 2 & \!\!-3 \\
1 &\!\!-1& 0 \end{array} \right]$, 
$\bar{P} = \left[ \begin{array}{crr}  
1 & 3 & 8 \\
1 & 3 & \!\!-4 \\
1 &\!\!-1& 0 \end{array} \right]$; the Galois fusion is the association scheme of the complete multipartite graph $\overline{3K_4}$.
\end{example}

Observe that, as a consequence, Theorems \ref{Thm:inclusion} and \ref{Thm:BijectiveDesignCollections} are always applicable for conjugacy class schemes.

\begin{problem}
    Does there exist a finite group $G$ whose rational fusion is a primitive 2-class association scheme? Note that some graphs with Paley parameters, e.g. $\mathsf{Paley}(q^2)$ but also $\mathrm{srg}(225,112,55,56)$, have rational eigenvalues.
\end{problem}

\begin{problem}
    Suppose $\EE$ and $\KK$ are subfields of the splitting field of the conjugacy class scheme $(G,\cR)$. Is there a simple relationship between $\EE$  and $\KK$ such that, whenever $(G,\cR)$ satisfies Property $\PropM{\EE}$, it must also satisfy Property $\PropM{\KK}$? 
\end{problem}

\section{Case study: Dicyclic groups}
\label{Sec:dicyclic}

The \emph{dicyclic group} $\Dic_n$ is the group of order $4n$ with presentation
$$ \Dic_n = \left\langle \ x,y \ \middle| \ x^{2n}=1, \ y^2=x^n, \ y^{-1}xy=x^{-1} \ \right\rangle.$$
For $n$ even, this group has the same character table as the dihedral group $D_{2n}$ of the same order. In fact, the two association schemes are isomorphic. Since the dihedral case is studied elsewhere, we will focus on the case where $n$ is odd.

Multiplication and conjugation (using $y^{-1}=yx^n$)  are given by 
\begin{center}
\hfill  \begin{tabular}{l|ll}
$g\cdot h$  &   $\phantom{y}x^\ell$  &   $yx^\ell$ \\ \hline
$\phantom{y}x^k$  &  $\phantom{y}x^{k+\ell}$  &  $yx^{\ell-k}$  \\
${y}x^k$  &  ${y}x^{k+\ell}$  &  $\phantom{y}x^{n+\ell-k}$  
\end{tabular} \hfill
\begin{tabular}{l|ll}
$h^g$  &   $\phantom{y}x^\ell$  &   $yx^\ell$ \\ \hline
$\phantom{y}x^k$  &  $\phantom{y}x^{\ell}$  &  $yx^{\ell+2k}$  \\
${y}x^k$  &  $\phantom{y}x^{-\ell}$  &  ${y}x^{2k-\ell}$  
\end{tabular}. \hfill  $\phantom{.}$
\end{center}

The  $n+2$ non-trivial conjugacy classes are
$\cC_k=\{x^k,x^{2n-k} \}$,  $1\le k< n$, $\cC_n=\{x^n\}$, 
$$\cC_{n+1} = \{ yx^{2k} : k=0,\ldots,n-1 \}, \quad  \cC_{n+2} = \{ yx^{2k+1} : k=0,\ldots,n-1 \}$$
where $\cC_{n+2} = \cC_{n+1}^{(-1)}$ since $n$ is assumed odd.

The conjugacy class association scheme contains two vertex-disjoint subschemes (or induced schemes) isomorphic to the association scheme of the $2n$-gon: one
induced on the cyclic subgroup $\langle x\rangle$ and the other on $y\langle x\rangle$. In both the association scheme of the dihedral group $D_{2n}$ 
and this association scheme, the remaining pairs are partitioned into two relations $R_{n+1}\cup R_{n+2} = \left( \langle x\rangle \times y\langle x\rangle\right) \cup  \left( y\langle x\rangle \times \langle x\rangle\right)$. Whereas, in the dihedral case,  these two relations are both isomorphic to two copies of the complete bipartite graph $K_{n,n}$,  in the dicyclic case the relations $R_{n+1}$ and
$R_{n+2}$ are orientations of $K_{2n,2n}$. More precisely, pairs in $R_{n+1}$ incident to $x^k \in \Dic_n$ are $(x^k,yx^{k+2\ell})$ and $(yx^{n-k+2\ell},x^k)$ ($0\le \ell < n$).

The (transpose of the) eigenmatrix for the association scheme of the cycle $C_{2n}$ appears as a submatrix of our
character table. The splitting field of $C_{2n}$ is well-known to be the real part of a cyclotomic extension, $\rats(\zeta_{2n}) \cap \re$, and we extend this by the square root of $-1$ to
get our splitting field $\FF = \left( \rats(\zeta_{2n}) \cap \re \right) (i) = \rats(\zeta_{2n}+\zeta_{2n}^{-1}, i)$.

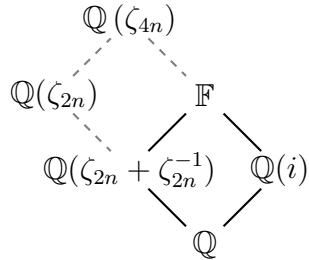
\begin{center}   
\begin{tikzpicture}[style=thick,solidvert/.style={   draw,  circle,   fill=black,  inner sep=1.2pt}]
\def \eps {0.25}
\draw[gray,dashed] (1+\eps,3-\eps) -- (2-\eps,2+\eps);
\draw (3+\eps,3-\eps) -- (4-\eps,2+\eps);
\draw (2+\eps,2-\eps) -- (3-\eps,1+\eps);
\draw[gray,dashed] (2+\eps,4-\eps) -- (3-\eps,3+\eps); 
\draw (3+\eps,1+\eps) -- (4-\eps,2-\eps);
\draw (2+\eps,2+\eps) -- (3-\eps,3-\eps);
\draw[gray,dashed] (1+\eps,3+\eps) -- (2-\eps,4-\eps); 
\node at (3,1) {$\mathbb{Q}$};
\node at (2,2) {$\mathbb{Q}(\zeta_{2n}+\zeta_{2n}^{-1})$};
\node at (4,2) {$\mathbb{Q}(i)$};
\node at (1,3) {$\mathbb{Q}(\zeta_{2n})$};
\node at (3,3) {$\mathbb{F}$};
\node at (2,4) {$\mathbb{Q}\left(\zeta_{4n}\right)$};
\end{tikzpicture}
\end{center}

The cyclotomic extension  $\rats(\zeta_m)$ has Galois group $\left( \ints_m\right)^\times$  \cite[Sec.~14.6]{df}; the Galois automorphism corresponding to $k\in \ints_m$ with $\gcd(k,m)=1$ is given by $\zeta_m\mapsto \zeta_m^k$. (See also \cite{conrad} for a nice exposition.) We have Galois group $\Gal(\FF/\rats)$ 
where our splitting field
$\FF$ is a subfield of $\rats(\zeta_{4n})$ whose Galois group is $\left( \ints_{4n}\right)^\times$. Since all subgroups are normal, all intermediate fields $\rats \subset \KK \subseteq \rats(\zeta_{4n})$
are Galois extensions.  So $[\FF:\rats]=[\rats(\zeta_{2n}):\rats]= \phi(2n)$, the Euler totient function. Since $n$ is assumed to be odd, $\FF$ is a degree $\phi(n)$ extension of the rationals.
By \cite[Prop.\ 14.19]{df},  $\Gal(\FF/\rats(i))\cong \Gal( \rats(\zeta_{2n}+\zeta_{2n}^{-1})/\rats)$, the Galois group of the splitting field of the $2n$-cycle,  
and  $\Gal(\FF/\rats(\zeta_{2n}+\zeta_{2n}^{-1})) \cong \ints_2$.  Moreover, writing $\zeta=\zeta_{4n}$ now, the subgroup of $\Gal\left( \rats(\zeta)/\rats \right)$ that stabilises $\FF$ is $N=\langle \zeta \mapsto \zeta^{2n-1} \rangle$ since $k=1$ and $k=2n-1$ are the only exponents with $\zeta^{nk}=\zeta^n$ and $\zeta^{2k}+\zeta^{-2k}=\zeta^2+\zeta^{-2}$. We now have $\Gal(\FF/\rats) \cong 
\left( \ints_{4n}\right)^\times/ \langle 2n-1\rangle$. So, for $1\le k\le 4n-1$ with $\gcd(k,4n)=1$, $\zeta\mapsto \zeta^k$ and $\zeta \mapsto \zeta^{2n-k}$ act identically on $\FF$ and we need only consider $0<k<n$ and $3n<k<4n$ or $-n<k<n$ with $\gcd(k,2n)=1$. 

With this notation for elements of $\Gal(\FF/\rats)$, let's see how $\bm{\theta}_k:\zeta \mapsto \zeta^k$ acts on the columns of matrix $Q$, i.e., on the rows of the character table. 
The character table for $\Dic_n$ takes the following general form.  Let $\kappa^{(r)} = \zeta^{2r}+\zeta^{-2r}=2\cos\left( \frac{\pi r}{n} \right)$ to save space.
\begin{center}
\begin{tabular}{c|ccccccc|rr}
$\cC$ & $\cC_0$ & $\cC_1$ & $\cC_2$ & $\cdots$ & $\cC_k$ & $\cdots$ & $\cC_n$ & $\cC_{n+1}$ & $\cC_{n+2}$ \\ 
$|\cC|$ & $1$ & $2$ & $2$ & $\cdots$ & $2$ & $\cdots$ & $1$ & $n$ & $n$ \\ \hline
$\chi_{0}$ & $1$ & $1$ & $1$ & $\cdots$ & $1$ & $\cdots$ & $1$ & $1$ & $1$ \\
$\chi_{1}$ & $1$ & $1$ & $1$ & $\cdots$ & $1$ & $\cdots$ & $1$ & $-1$ & $-1$ \\
$\chi_{2}$ & $1$ & $\!\!\!\!-1$ & $1$ & $\cdots$ & $(-1)^k$ & $\cdots$ & $(-1)^n$ & $i$ & $-i$ \\
$\chi_{3}$ & $1$ & $\!\!\!\!-1$ & $1$ & $\cdots$ & $(-1)^k$ & $\cdots$ & $(-1)^n$ & $-i$ & $i$ \\ \hline
$\psi_{1}$ & $2$ & $\kappa^{(1)}$ & $\kappa^{(2)}$ & $\cdots$ & $\kappa^{(k)}$ & $\cdots$ & $\!\!\!\!-2$ & $0$ & $0$ \\
$\psi_{2}$ & $2$ & $\kappa^{(2)}$ & $\kappa^{(4)}$ & $\cdots$ & $\kappa^{(2k)}$ & $\cdots$ & $2$ & $0$ & $0$ \\
$\vdots$ & $\vdots$ & $\vdots$ & $\vdots$ & $\ddots$ & $\vdots$ & $\ddots$ & $\vdots$ & $\vdots$ & $\vdots$ \\
$\psi_{n-1}$ & $2$ & $\kappa^{(n-1)}$ & $\kappa^{(2n-2)}$ & $\cdots$ & $\kappa^{(nk-k)}$ & $\cdots$ & $(-1)^{n-1}2$ & $0$ & $0$ \\
\end{tabular}
\end{center}

We compute $\bm{\theta}_k\left( \kappa^{(r)}\right) = \bm{\theta}_k \left( \zeta^{2r}+\zeta^{-2r}\right)=  \zeta^{2kr}+\zeta^{-2kr}= \kappa^{(kr)}$ when $kr<n$. 
As this suggests,  $\bm{\theta}_k$ and $\bm{\theta}_{-k}$  have the same action on $\kappa^{(1)},\ldots,\kappa^{(n-1)}$, 
but one fixes $\pm i$ while the other swaps them, namely the one among $\pm k$ that is three 
modulo four. 

If, for integers $k,m$ with $m>1$,  we write $k \mmod m = k'$ for that unique $k' \equiv k \pmod{m}$ in the range $-\frac{m}{2} < k' \le \frac{m}{2}$, then we have
$$ \bm{\theta}_k : \psi_\ell \mapsto \psi_{\ell^{\theta_k}} = \psi_{|k\ell \mmod 2n|}~.$$
So each permutation $\theta_k \in \Sigma_\FF$ 
defined by $E_\ell^{\bm{\theta}_k} = E_{\ell^{\theta_k}}$ fixes $0$ and $1$, fixes  (resp., swaps) 
$2$ and $3$ if $k\equiv 1 \pmod{4}$ (resp., $k\equiv 3 \pmod{4}$), and for $E_{j+3}$ corresponding to $\psi_{j}$,  $\theta_k(j+3)= |kj \mmod 2n|+3$.

We can simplify things significantly if we merely wish to work out the orbits of  $\Sigma_\rats$ 
on $\{0,1,\ldots,d\}$. We have orbits $\cQ_0=\{0\}$, $\cQ_1=\{1\}$, $\cQ_2=\{2,3\}$ and the remaining orbits are precisely the orbits of the Galois group of the $2n$-cycle acting on $\{ \psi_1,\ldots, \psi_{n-1}\}$ as the transpose of the eigenmatrix of this (self-dual) scheme appears among the first $n+1$ columns of $\sT$ after deleting repeated rows.

So our rational decomposition tool gives the following result in this case. 

\begin{corollary}
\label{Cor:clumpsDicyclic}
Let $(\Dic_n,\cR)$ be the conjugacy class scheme of the dicyclic group $\Dic_n$ with primitive idempotents indexed $\{E_{\chi_0},
E_{\chi_1},E_{\chi_2},E_{\chi_3},E_{\psi_1},\ldots,E_{\psi_{n-1}}\}$ as above. Let $C\subset \Dic_n$ be a Delsarte $T$-design for $T=T(C)$. 
Then $\chi_2\in T$ if and only if $\chi_3\in T$ and, for $1\le r,k<n$ with $\gcd(k,n)=1$, $\psi_\ell \in T$ if and only if $\psi_{|k\ell \mmod 2n|}\in T$ for each $-n<k<n$ with $\gcd(k,2n)=1$. \hfill $\Box$
\end{corollary}

Note that, since $\gcd( \ell^{\theta_k},2n)$ is divisible by $\gcd(\ell,2n)$, the set $\psi_\ell, \psi_{2\ell},\psi_{3\ell},\ldots$ is a union of Galois orbits.


Every subgroup of $\Dic_n$ is either a cyclic or dicyclic. Since we are restricting to the case where $n$ is odd, $k$ is even if and only if $2n/k$ is odd. For $H=\langle x^k\rangle$, we have $|H|= 2n/k$ and $H \unlhd \Dic_n$ with inner distribution  of the form
 $a = [1,0,\ldots,0,2,0,\ldots,0,2,0,\ldots \ldots, 0,0]$
where $a_i=2$ if $k|i$ and $i<n$, $a_n=1$ if $k$ is odd and $a_n=0$ if $k$ is even. (The remaining entries of $a$ are zero.)  The dual distribution $b=aQ$ has entries naturally indexed by $\chi_0,\chi_1,\chi_2,\chi_3,\psi_1,\ldots,\psi_{n-1}$ with values
\begin{center}
    \begin{tabular}{|c|cc|cc|c|l|} \hline
    $\ell=2n/k$ & $\chi_0$,&$\chi_1$ & $\chi_2$,&$\chi_3$ & $\psi_{\ell i}$ ($1\le i \le \lfloor (k-1)/2 \rfloor$) &  $\psi_j$ (otherwise) \\ \hline
    $|H|=\ell$ even & $\ell$ & $\ell$  & $0$ & $0$  &   $4\ell$ & 0 \\
    $|H|=\ell$ \ odd & $\ell$ & $\ell$ & $\ell$  & $\ell$  &   $4\ell$ & 0 \\    \hline
    \end{tabular}
\end{center}

So the idempotents outside $T(H)$ are either $\{E_{\chi_0},
E_{\chi_1},E_{\psi_\ell},E_{\psi_{2\ell}},E_{\psi_{3\ell}},\ldots\}$  or $\{E_{\chi_0},
E_{\chi_1}$, $E_{\chi_2}$ ,$E_{\chi_3},E_{\psi_\ell},E_{\psi_{2\ell}},E_{\psi_{3\ell}},\ldots\}$.

The case where $H$ is dicyclic is almost identical. Here the cyclic subgroup $H \cap \langle x\rangle$ behaves as above and $H \setminus \langle x\rangle$ contributes either $a_{d-1}=|H|/2=2n/k$
when $k$ is even or 
$a_{d-1}=a_d=|H|/4=n/k$ when $k$ is odd. In the first case, the first few entries of the MacWilliams transform are
$aQ = \left[ \begin{array}{ccccc}
\frac{4n}{k} & 0 & 0 & 0 & \cdots
\end{array} \right]$ and, in the second case, $aQ = \left[ \begin{array}{ccccc}
\frac{4n}{k} & 0 & \frac{4n}{k} & \frac{4n}{k} & \cdots
\end{array} \right]$; in both cases, the entries of $aQ$ indexed by $\psi_1,\ldots,\psi_{n-1}$ are exactly as in the table above for the cyclic subgroups.

\section*{Acknowledgments}
WJM is supported through a grant 
from the National Science Foundation (DMS Award \#1808376) which is gratefully acknowledged. The major part of this work was carried out while WJM was visiting 
the University of Canterbury generously supported by an Erskine Fellowship; he thanks the School of Mathematics and Statistics, University of Canterbury, for their hospitality during this stay.
The second author thanks Sycamore Herlihy for helpful conversations.

\bigskip

\noindent \textsl{Note added in proof.} In the final days of editing, the authors learned of the preprint \cite{GZ} of Godsil and Zhang which also addresses fusions of group schemes in the context of continuous time quantum walks on Cayley graphs. Whereas the present paper considers the span of the idempotents in $\BMA \cap \Mat_X(\KK)$, the approach in \cite{GZ} considers the matrices in $\BMA \cap \Mat_X(\KK)$ whose eigenvalues also lie in $\KK$.

\end{document}